\documentclass{IEEEtran}
\usepackage{amsmath}    
\usepackage{graphics,epsfig,subfigure}    
\usepackage{verbatim}   
\usepackage{color}      
\usepackage{subfigure}  
\usepackage{hyperref}   
\usepackage{amssymb}
\usepackage{epstopdf}
\usepackage{graphicx}
\usepackage{wasysym}
\usepackage{wrapfig}
\usepackage{sidecap}
\usepackage{float}
\usepackage{subfigure}
\newcommand{\expect}[1]{\mathbb{E}\bigg[#1\bigg]}
\newcommand{\expectm}[1]{\mathbb{E}\big[#1\big]}

\newcommand{\argmax}[1]{\underset{#1}{\mbox{arg max }}}
\newcommand{\bv}[1]{{\boldsymbol{#1} }}
\newcommand{\script}[1]{{{\cal{#1} }}}

\newcommand{\tsf}[1]{\textsf{#1}}

\allowdisplaybreaks
\begin{document}
\title{Optimal Power Procurement and Demand Response with Quality-of-Usage Guarantees}

\abovedisplayskip=.05in \belowdisplayskip=.05in
\author{\large{Longbo Huang, Jean Walrand, Kannan Ramchandran}%
\thanks{The authors are with the EECS dept at UC Berkeley, Berkeley, CA, 94720, USA. Emails: \{huang, wlr, kannanr\}@eecs.berkeley.edu. }
} 

\date{}
\maketitle

\newtheorem{remark}{Remark}
\newtheorem{fact_def}{\textbf{Fact}}
\newtheorem{coro}{\textbf{Corollary}}
\newtheorem{lemma}{\textbf{Lemma}}
\newtheorem{main}{\textbf{Proposition}}
\newtheorem{theorem}{\textbf{Theorem}}
\newtheorem{claim}{\emph{Claim}}
\newtheorem{prop}{Proposition}
\newtheorem{assumption}{\textbf{Assumption}}
\newtheorem{condition}{\textbf{Condition}}

\begin{abstract}
In this paper, we propose a general operating scheme which allows the utility company to jointly perform power procurement and demand response so as to maximize the social welfare. Our model takes into consideration the effect of the renewable energy and the multi-stage feature of the power procurement process. It also enables the utility company to provide quality-of-usage (QoU)  guarantee to the power consumers, which ensures that the average power usage level meets the target value for each user. 
To maximize the social welfare, we develop a low-complexity algorithm called the \emph{welfare maximization algorithm} (\tsf{WMA}), which performs joint power procurement and dynamic pricing. \tsf{WMA} is constructed based on a two-timescale Lyapunov optimization technique. We prove that \tsf{WMA} achieves a close-to-optimal utility and ensures that the QoU requirement is met with bounded deficit. 
\tsf{WMA} can be implemented in a distributed manner and is robust with respect to system dynamics uncertainty. 
\end{abstract}

\vspace{-.1in}
\section{Introduction}
Recent years have witnessed enormous efforts to modernize the power grid. 
Central to this challenging task are the integration of renewable energy technologies \cite{der-energy-gov02} and the design of efficient user demand-response schemes \cite{dr-report}. 
In this paper, we propose a general operating scheme which allows the utility companies to optimally integrate the available renewable energy sources to maximize the social welfare, and to perform optimal demand response to ensure   the users' quality-of-usage (QoU).
%

Specifically, we consider a utility company, equipped with a renewable energy source, purchasing power from a two-stage electricity market to meet the users' demand. 
Every day, the utility company first purchases power from the day-ahead market for the next day, called the base-power, based on its forecast of the user demand and the available renewable energy. 
Then, if the base-power and the renewable power are enough to meet the user demand on the next day, no further action is needed; otherwise, the utility company purchases additional power from the real-time market to meet the power deficit, which incurs an extra cost. 
Besides performing power procurement, the utility company also adjusts its electricity selling prices to encourage the users to shift their deferrable loads to times when the prices in the market are low. At the same time, the utility company  also has to ensure the users' quality-of-usage (QoU), i.e., \emph{the average user power consumption level exceeds a pre-agreed level.} 
%
%
The objective of the utility company is to optimally perform power procurement and determine the electricity prices, so as to maximize the social welfare, i.e., aggregate user utility minus the cost for supporting the user loads, subject to  guaranteeing the users' QoU. 

There have been many previous works on integrating renewable energy into the power grid, designing optimal power procurement algorithms, and constructing demand response schemes. 
\cite{pravin-risk11} proposes the concept of ``risk-limiting-dispatching'' as a way of incorporating renewable energy into the current power grid. 
Works \cite{eilyan-sellingwind} and \cite{eilyan-tps11} consider the problem of selling wind power in the current electricity market. 
 \cite{cai-contract-cdc-11} designs contract schemes to best  commercialize renewable energy for the electricity market. 
On the power procurement side, \cite{libin-cdc-11} and \cite{libin-allerton-11}  consider the problem of optimally procuring power to maximize the social welfare for a single day under both independent and time-correlated demand conditions,  using stochastic subgradient methods. 
\cite{poor-power-11} considers the problem of minimizing the expected power cost and formulates the problem as a Markov decision problem. 
\cite{neely-grid-10} considers the problem of optimally utilizing the renewable energy for both cost minimization and revenue maximization. \cite{he-multi-time-wind-11} formulates the optimal power dispatch problem as a Markov decision problem. 
On the demand response side, \cite{nali-dr-11} and \cite{lijun-dr-book-11} formulate the problem of optimal demand response as convex programs and study the role of dynamic pricing. \cite{rad-home-11} uses a game theoretic approach to study the demand management problem. 
However, we note that most of the previous works (a) either consider the three problems  separately, in which case the resulting algorithm may not be optimal for the entire system, or (b) do not consider the multi-stage nature of the power procurement process, which makes the solution less applicable to the    power grid, 
or (c) only consider optimizing the system for a single period, in which case finding the optimal solution relies heavily on the convexity assumptions of the user utility functions. 



In this paper, we consider the problem of jointly optimizing power procurement and demand response to maximize the long term average social welfare. 
Our model fully accounts for the multi-stage nature of the electricity market and allows the user utility functions to be \emph{arbitrary} increasing functions. 
%
Our algorithm design approach uses a two-stage Lyapunov optimization technique, which allows us to develop low-complexity optimal control algorithms. 
The main contributions of the paper are summarized as follows:
\begin{itemize}
\item A general operating scheme for utility companies to jointly 
optimize power procurement and demand response to maximize the long term system welfares and ensure the users'  quality-of-usage (QoU).  
\item The optimal power procurement scheme for utility companies to maximize the gain of  renewable energy sources. Explicitly quantified  the value of renewable energy to the system, and the effects of its mean and variance on the gain. 
\item Characterization of the social welfare loss due to the constraint that all users must see the same price. 
\item A low-complexity online control algorithm welfare maximization algorithm (\tsf{WMA}),  which fully accounts for the multi-stage nature of the electricity market, and is  provably near-optimal for maximizing the social welfare while ensuring that the users' QoU are met with \emph{bounded} deficit for all times. 
%
\item An evaluation of our scheme using real-world renewable energy and market power prices  data, and a demonstration that it effectively maximizes the social welfare. 
\end{itemize}

This paper is organized as follows. In Section \ref{section:model}, we describe our system model. In Section \ref{section:alg}, we present our algorithm design approach and construct the welfare maximization algorithm (\tsf{WMA})  to jointly optimize power procurement and pricing. 
We then analyze our scheme in Section \ref{section:analysis}. Numerical results are presented in Section \ref{section:sim}. We conclude our paper in Section \ref{section:conclusion}. 




%

\section{System model}\label{section:model}
We consider a system shown in Fig. \ref{fig:model}. In the system, 
the utility company is supplying power to $N$ users. 
Time is divided into days, and each day is further divided into $T$ equal size slots that correspond, e.g.,  to one-hour l or $15$-minute intervals. 

%

\begin{figure}[cht]
\vspace{-.1in}
\centering
\includegraphics[height=2.2in, width=3.5in] {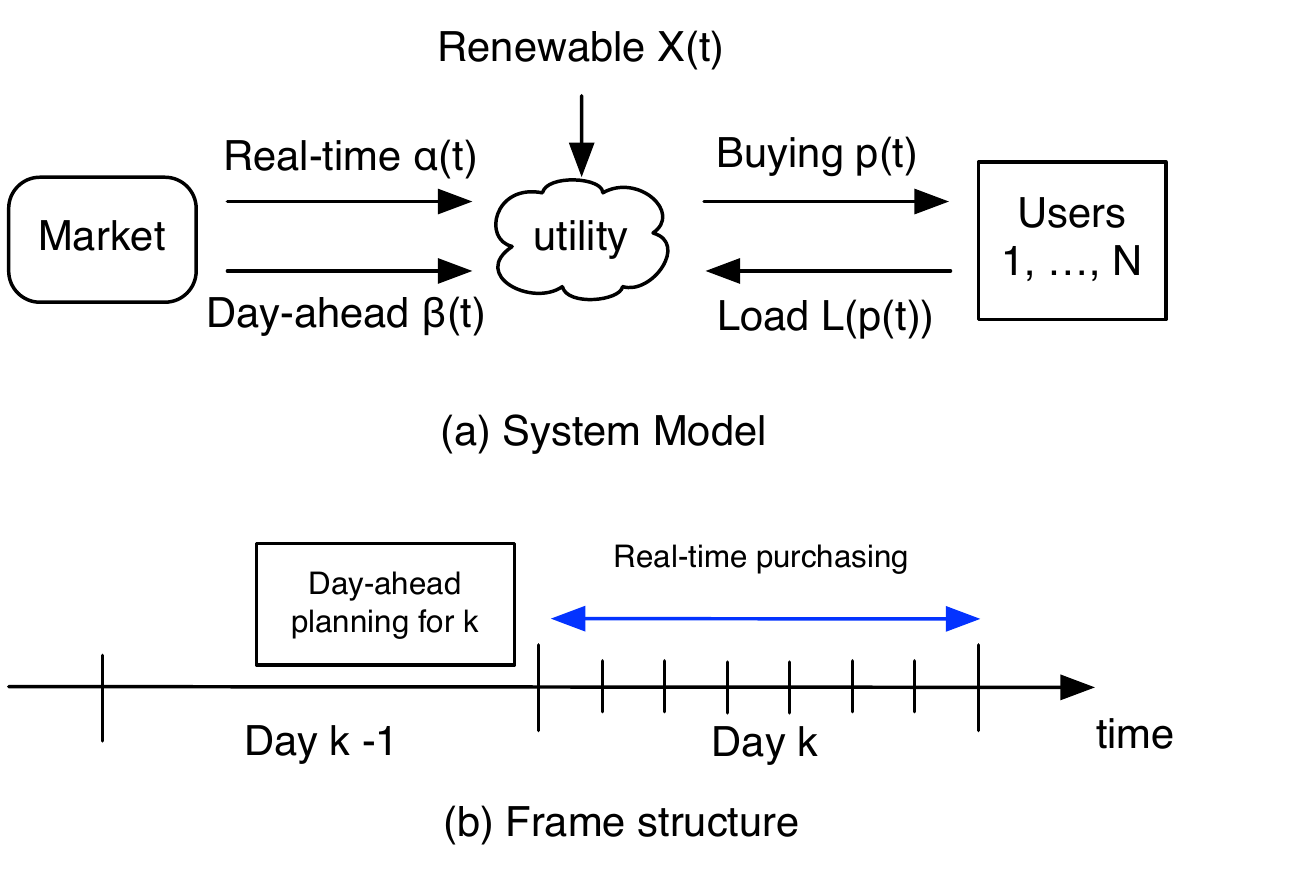}
\vspace{-.2in}
\caption{\small{(a) The system model. (b) The time structure. }}
\label{fig:model}
\vspace{-.2in}
\end{figure}

\subsection{Power procurement model}
We assume that the utility company has a renewable energy source and denote the renewable power it obtains in slot $t$ on day $k$ by $X(k, t)$. This renewable energy source can be solar panels or wind turbines owned by the utility company. We assume that $X(k, t)\stackrel{d}{=}X(t)$ for some random variable $X(t)$ and are i.i.d. every day. \footnote{This assumption can be viewed as ignoring the potential time dependence of $X(k, t)$ on $k$, e.g., the seasonal behavior change of wind power. However, our results can be extended to incorporate such dependence on $k$. } Here ``$\stackrel{d}{=}$''  represents ``equal in distribution.'' 
We denote the probability density function (p.d.f.) of $X(t)$ by $f_X(x, t)$, and  compactly write the renewable energy as $X(t)$ in the following. 
We denote the maximum capacity of the renewable energy by $x_{\tsf{max}}$, so that $X(t)\leq x_{\tsf{max}}$ with probability one. 

To serve the users, the utility company operates according to the following procedure. \footnote{In practice, there will be ancillary services in addition to the two stages considered here. Our model can be extended to incorporate these components. }

\textbf{Demand response:} On day $k-1$, the utility company first chooses the  power price vector  $\bv{p}(k)=(p(k, t), t=0, 1, ..., T-1)$ for day $k$ from some feasible price set $\script{P}$, and announces the prices. 
We assume that the feasible price set $\script{P}$ is a time-invariant and compact subset of $\mathbb{R}^{T}$, which includes the constraint $0\leq p(k, t)\leq p_{\tsf{max}}$ for all $k, t$. For instance, one example is $\script{P}=[0, p_{\max}]^T$. 
%

 \textbf{Day-ahead planning:} %
On day $k-1$, the utility company first forecasts the aggregate user demand $L(k, t)$ based on its knowledge of the users' responses to the prices. We assume that the utility company can accurately  forecast $L(k, t)$, including its mean and the distribution of the randomness. \footnote{Note that this is not very restrictive and can be done when the user utility functions are known. Even in the case when they are not known, the utility company can still predict with empirical data. }

%
To efficiently utilize the renewable energy, the utility company also forecasts the available renewable energy $X(t)$ on day $k$ based on $f_X(x, t)$. Then, it procures a base-power $B(k, t)$ from the day-ahead market based on the estimated load $L(k, t)$ and renewable energy $X(t)$, at a price  $\beta(k, t)$. 

\textbf{Real-time purchasing:} At time $t$ on day $k$, if the procured power and the renewable energy  are enough to meet the demand, i.e., $B(k, t)+X(t)\geq L(k, t)$,  no further action is needed. 
Otherwise, the utility company has to purchase additional power from the real-time market with a real-time price $\alpha(k, t)$ to meet the power deficit $Y(k, t)\triangleq L(k, t) - B(k, t)- X(t)$. 
%


%

%
In the above steps, we assume that when purchasing power from the electricity market on day $k-1$, the utility company can observe the day-ahead price $\beta(k, t)$ and the expected real-time $\bar{\alpha}(k, t)$. 
In practice, the day-ahead and real-time prices are time-varying. To model this aspect, we denote $\bv{\beta}(k)=(\beta(k, 1), ..., \beta(k, T))$ and $\bar{\bv{\alpha}}(k) =(\expectm{\alpha(k, 1)}, ..., \expectm{\alpha(k, T)})$, the market price vector pair.  
%
%
We then first assume that on each day, the pair $(\bv{\beta}(k), \bar{\bv{\alpha}}(k))$ is drawn i.i.d. from some finite market price set $\script{P}^{\tsf{m}}=\{(\bv{\beta}_j, \bar{\bv{\alpha}}_j), j=1, ..., M\}$ according to some distribution which may not be known to the utility company.  
We will later extend our results to the case when the vector evolves according to some finite state Markovian process. 
Finally, we assume that both the day-ahead and real-time prices are upper bounded, i.e., $0\leq \alpha(k, t)\leq \alpha_{\tsf{max}}$ and  $0\leq\beta(k, t)\leq \beta_{\tsf{max}}$, and denote $\delta_{\tsf{max}} \triangleq \max[\alpha_{\tsf{max}}, \beta_{\tsf{max}}]$. 

\subsection{User demand model}

We assume that each user derives a utility $U_n(k, t)=U_n(L, t)$ at time $t$ on day $k$ by consuming a load of $L_n^{\tsf{min}}(t)\leq L\leq L_n^{\tsf{max}}$, where $L_n^{\tsf{min}}(t)$ is the minimum amount of load that user $n$ must consume at time $t$, e.g., lighting at night. \footnote{Our results readily extend to the case when $U(\cdot, \cdot)$ is also a function of $k$, in which case each day has  different utility functions for each time. } 
Each utility function $U_n(L,  t)$ is only assumed to be an increasing and continuous function of $L$. 

Upon receiving the prices for day $k$, each user first plans an \emph{intended} power consumption level $L^{\tsf{d}}_n(k, t)$ for each time $t$ (the choice of $L^{\tsf{d}}_n(k, t)$ will be specified later). 
However,  the \emph{actual} power consumption on day $k$ may not be exactly equal to $L^{\tsf{d}}_n(k, t)$. We model this effect by assuming that the actual consumption is equal to $L_n(k, t)=L^{\tsf{d}}_n(k, t)+w_n(k, t)$ for some i.i.d. zero mean random variable $w_n(k, t)$. We assume that $w_n(k, t)\stackrel{d}{=} w_n(t)$ for some random variable $w_n(t)\leq w_{n}^{\tsf{max}}$ for all time. We denote the p.d.f. of $w_n(t)$ by $f_{w_n}(t)$ and  assume that $w_n(t) + L^{\tsf{d}}_n(k, t)\geq L_n^{\tsf{min}}(t)$ for all times, and that $w_n(t)$ is independent of the renewable energy $X(t)$. \footnote{Note that in practice, user $n$'s power consumption may well be correlated with the renewable energy level. This effect is captured by the utility function $U_n(L, t)$ in our formulation.}

%

Each user determines his intended consumption by solving the following maximization problem, where the expectation is taken over $w_n(k, t)$: 
%
\begin{eqnarray}
\hspace{-.3in}&&L^{\tsf{d}}_n(k, t)= \min\argmax{L_n^{\tsf{min}}(t)\leq L\leq L_n^{\tsf{d, max}}  }  \expect{U_n(L+w_n(k, t), t) \label{eq:user-choice}\\
\hspace{-.3in}&&\qquad\qquad\qquad\qquad\qquad\qquad\quad  - p(k, t) (L+w_n(k, t))}.\nonumber
\end{eqnarray}
Here $L_n^{\tsf{d, max}} = L_n^{\tsf{max}} -w_n^{\tsf{max}}$ is the maximum planned consumption level. 
That is, the user picks his load to minimize the expected utility minus cost. In the case when there are multiple choices, he picks the load with minimum magnitude. Therefore, the price vector uniquely determines the loads of the users. 
We also notice that, since $U_n(L, t)$ is a function of a single variable $L$, the maximization problem of user $n$ can usually be efficiently solved with high accuracy, even if $U_n(L, t)$ is nonconcave. 

In practice,  besides maximizing his utility via power consumption every time slot, each user also has a requirement on the average power usage level. To model this aspect, we assume that each user $n$ has the following 
\emph{quality-of-usage} (QoU) requirement:  \footnote{In our problem, this condition can be shown to be equivalent to requiring that the constraint holds with probability $1$ when the limit is well defined.}
 \begin{eqnarray}
\liminf_{K\rightarrow\infty}\frac{1}{KT}\sum_{k=0}^{K-1}\sum_{t=0}^{T-1}\expectm{L_n(k, t)} \geq L_n^{\tsf{av}}. \label{eq:min-def-load}
\end{eqnarray}
%
%
%
%
Note that $L_n^{\tsf{av}}$ can also be viewed as measuring the willingness of user $n$ to perform load shifting:  a smaller $L_n^{\tsf{av}}$ represents less stringent QoU requirement, and the user is more willing to defer his load. 
%
%
The constraint (\ref{eq:min-def-load}) can be viewed as an approximate  version of the more stringent constraint $\sum_{t=0}^{T-1}L_n(k, t) \geq TL_n^{\tsf{av}}$ for each frame $k$, which inevitably requires dynamic-programming as the solution technique \cite{libin-cdc-11}. 
As we will see later,  (\ref{eq:min-def-load}) allows us to approximately solve the harder constraint with a much simpler online algorithm. 

Finally, to model the user heterogeneity, we assume that there exists a positive constant $\gamma\geq1$ such that: 
for any given price $p(k, t)$ and for any user pair $n$ and $m$, the \emph{intended} consumptions are bounded by each other: 
\begin{eqnarray}
L^{\tsf{d}}_n(k, t) \leq \gamma L^{\tsf{d}}_m(k, t), \,\, L^{\tsf{d}}_m(k, t)\leq \gamma L^{\tsf{d}}_n(k, t). \label{eq:hetero}
\end{eqnarray}
When $\gamma=1$, all the users are homogeneous; whereas if $\gamma=\infty$, the users can be highly heterogeneous. 
Now by defining $w(t)\triangleq\sum_nw_n(t)$ and $L^{\tsf{d}}(k, t)\triangleq \sum_n L^{\tsf{d}}_n(k, t)$ to be the random and deterministic components of the load, 
 the aggregate load can be expressed as:
\begin{eqnarray}
L(k, t) = L^{\tsf{d}}(k, t) + w(k, t). \label{eq:ag-load}
\end{eqnarray}

%
%



%
%

%
%
%

\vspace{-.1in}
\subsection{Utility company objective}
On day $k$, the cost for the utility company at time $t$ is given by:
\begin{eqnarray}
  \text{Cost}(k, t) \triangleq \beta(k, t) B(k, t) + \alpha(k, t) Y(k, t),  \label{eq:power-cost}
\end{eqnarray}
where $Y(k, t)=\big[  L(k, t) - B(k, t)-X(t) \big]^+$ denotes the power deficit at time $t$. 
%
%
%
%
The welfare obtained by the utility company-user system on day $k$ is given by: 
\begin{eqnarray}
\text{Welfare}(k) \triangleq \sum_{t=0}^{T-1}\bigg[\sum_{n=1}^N U_{n}(L_n(k, t), t) - \text{Cost}(k, t)\bigg]. \label{eq:welfare-def}
\end{eqnarray}
Now for any feasibly power procurement and dynamic pricing policy $\Pi$, its  \emph{expected} average social welfare is defined: 
\begin{eqnarray}
\text{Welfare}^{\Pi}_{\tsf{av}} \triangleq \liminf_{K\rightarrow\infty}\frac{1}{K}\sum_{k=0}^{K-1}\expect{\text{Welfare}^{\Pi}(k)}. \label{eq:avg-welfare}
\end{eqnarray}
Here the expectation is taken over the random renewable energy $X(t)$, the random market prices and $w_n(t)$ for each user. 
%

%

Below, we define $\text{Welfare}^*_{\tsf{av}}$ to be the optimal average welfare over all feasible policies that ensure the QoU constraint (\ref{eq:min-def-load}). 
The objective of the utility company is to find an optimal dynamic pricing and power procurement policy to maximize (\ref{eq:avg-welfare}) subject to (\ref{eq:min-def-load}). 
In the following, we refer to this problem as the \emph{Welfare Maximization Problem} (WMP). 
To ensure the existence of at least one optimal feasible control policy, we assume that: 
\begin{eqnarray}
L_n^{\tsf{av}} > \max_{0\leq t\leq T-1} L_n^{\tsf{min}}(t), \quad L_n^{\tsf{d, max}} \geq L_{n}^{\tsf{av}} + w_n^{\tsf{max}},  \,\,\forall\,n. \label{eq:load-cond}
\end{eqnarray}
Note that the constraints in (\ref{eq:load-cond}) simply say that the QoU level of each user is no less than the minimum requirement and is no more than the maximum consumption. 
We also assume that both the renewable energy $X(t)$ and the load randomnesses $w_n(t)$ are continuous random variables.

%
%
%
%


\section{Algorithm design}\label{section:alg}
In this section, we construct algorithms to solve the WMP problem. Our solution is based on the Lyapunov network optimization technique \cite{neelynowbook} and requires minimum information of the random dynamics in the system. 
%
%
%
To carry out our construction, we first present in Section \ref{subsection:base-power} the optimal base-power procurement component of our algorithm, which maximizes the gain brought about to the utility company by the renewable energy. This result is not only interesting of its own, but also enables us to convert the 
$\text{Welfare}(k)$ function into a function only of the user prices $p(k, t)$, which facilitates the design of the dynamic response algorithm in Section \ref{subsection:alg-design}.  



\subsection{Optimal base-power procurement}\label{subsection:base-power}
In this section, we study how the utility company can maximize the gain of the renewable energy. 
%
To do so, we first define: 
\begin{eqnarray}
Z(t) \triangleq  X(t) - w(t), \label{eq:load-minus-ran}
\end{eqnarray}
as the ``effective'' random renewable energy, and use $f_Z(z, t)$ and $F_Z(z, t)$ to denote the p.d.f. and  c.d.f. of $Z(t)$. 
We also use $F^{-1}_Z(a, t)$ to denote the \emph{inverse} c.d.f. function of $Z(t)$, i.e., 
\begin{eqnarray}
F^{-1}_Z(a, t)= \inf\{z: F_Z(z, t) = a\}. 
\end{eqnarray}
We now state the following lemma, which gives the optimal power procurement scheme for the utility company to minimize its \emph{expected} cost for meeting a load $L(k, t)$ at time $t$ on day $k$. 
\begin{lemma}\label{lemma:value-renewable}
The optimal base-power $B^*(k, t)$ is given by: 
\begin{equation} \label{eq:opt-B-def}
B^*(k, t) = \left\{\begin{array}{rl}
0,\qquad\qquad\qquad & \text{if } \frac{\bar{\alpha}(k, t)}{\beta(k, t)}<1\\
\bigg[L^{\tsf{d}}(k, t)-F_Z^{-1}(\frac{\beta(k, t)}{\bar{\alpha}(k, t)}, t)\bigg]^+, & \text{else}. 
\end{array}\right. 
\end{equation}
When $B^*(k, t)=0$, we have: 
\begin{eqnarray}
\expectm{\text{Cost}^*(k, t)} = \bar{\alpha}(k, t) \mathbb{E}\big[ (L^{\tsf{d}}(k, t) - Z(t))^+ \big]. \label{eq:cost-zero-B}
\end{eqnarray}
Else. if $B^*(k, t)>0$, we have: 
\begin{eqnarray}
&& \hspace{- 0.75in} \expectm{\text{Cost}^*(k, t)} =  \beta(k, t) L^{\tsf{d}}(k, t)  \nonumber  \\
&& \hspace{0.2in}  - \bar{\alpha}(k, t) \int_{-\infty}^{F_Z^{-1}(\frac{\beta(k, t)}{\bar{\alpha}(k, t)}, t)} zf_Z(z, t)dz.  \label{eq:value-renewable}
\end{eqnarray}
Moreover, $\expectm{\text{Cost}(k, t)}$ is non-decreasing in $L_n^{\tsf{d}}(k, t)$.  $\Diamond$
\end{lemma}
\begin{proof}
See Appendix A. 
\end{proof}
Note that the second term on the right-hand-side (RHS) in (\ref{eq:value-renewable}) can indeed be interpreted as the \emph{value of the renewable energy} $X(t)$ (denoted as  \text{VoR}(t)) (without $X(t)$, the minimum expected cost is $\beta(k, t)L^d(k, t)$), i.e., 
\begin{eqnarray}
\text{VoR}(t) = \bar{\alpha}(k, t) \int_{-\infty}^{F_Z^{-1}(\frac{\beta(k, t)}{\bar{\alpha}(k, t)}, t)} zf_Z(z, t)dz. \label{eq:value-of-renewable}
\end{eqnarray}
%
%
%
The following corollary follows from Lemma \ref{lemma:value-renewable} and shows that the value of the renewable energy is increasing with its mean and non-increasing with its variance. 
\begin{lemma}\label{lemma:value-of-renewable}
The value of renewable energy $\text{VoR}(t)$ is increasing with the mean of $X(t)$ and nondecreasing with the variance of $X(t)$. 
\end{lemma}
\begin{proof}
See Appendix B. 
\end{proof}

%




\subsection{The Lyapunov approach}\label{subsection:alg-design}
We now present the construction of the joint pricing-power procurement algorithm using the Lyapunov technique. The method works as follows: we first define a set of ``deficit'' queues for each of the constraints we want to guarantee, i.e., (\ref{eq:min-def-load}). Then, using the queueing dynamics (defined later), we obtain a drift inequality, which captures how the actions at every time affect the system welfare and the deficit queue sizes. 
After that, we construct our algorithm by finding our actions to minimize  RHS of the drift inequality, which corresponds to finding the ``descent'' direction for the system state. %
By doing so at every time slot, we can guarantee that the average welfare is maximized. 
Also, although we start with the time average constraint (\ref{eq:min-def-load}), we can get as a by-product of our algorithm an explicit \emph{sample-path} deficit queueing bound. This provides explicit delay guarantee for the deferred loads of the users, and is very useful in practice. 
%
%


To start, we define a \emph{load-deficit} queue $Q_n(\tau), \tau=0, 1, ...$ for each user $n$ as follows: $Q_n(0)=0$, and it evolves according to: 
\begin{eqnarray}
\hspace{-.02in}Q_n(t_k+t+1) = [ Q_n(t_k+t) - L_n(k, t)]^+ + L^{\tsf{av}}_n,\label{eq:queue-dynamic}
\end{eqnarray}
for all $k$ and all $t$. Here $t_k\triangleq kT$ is introduced for notation simplicity. 
The intuition behind the deficit queue is that, if we can guarantee $Q_n(t)/t\rightarrow0$,  then (\ref{eq:min-def-load}) will be ensured. 
To see this, note from (\ref{eq:queue-dynamic}) that:
\begin{eqnarray}
Q_n(\tau+1) \geq Q_n(\tau) - L_n(\tau) + L^{\tsf{av}}_n. 
\end{eqnarray}
Here $\tau=t_k+t$. Taking expectations on both sides and summing it up from $\tau=0, ... r-1$, we get: 
\begin{eqnarray}
\expectm{Q_n(r)} \geq \expectm{Q_n(0)} - \sum_{\tau=0}^{r-1}\expectm{L_n(\tau)} + rL^{\tsf{av}}_n. 
\end{eqnarray}
Dividing both sides by $r$, and taking the liminf as $r\rightarrow\infty$, we see that (\ref{eq:min-def-load}) follows. 

Now to make use of the deficit queues, we define a Lyapunov function $V(\tau) \triangleq \frac{1}{2}\sum_{n=1}^NQ^2(\tau)$ and denote $\bv{Q}(\tau)=(Q_n(\tau), n=1, ..., N)$. We then define the following $T$-slot Lyapunov drift for every time $t_k=kT$: 
\begin{eqnarray}
\Delta(t_k)  = \expect{V(t_{k+1}) - V(t_k)\left.|\right. \bv{Q}(t_k)}. \label{eq:drift-def}
\end{eqnarray}
Here the expectation is taken over the randomness of the renewable energy and the potentially random power procurement and pricing action selections.

We now proceed to construct our algorithm. We first have the following lemma regarding the drift defined in (\ref{eq:drift-def}). 
\begin{lemma}\label{lemma:pure-drift}
For any value $k=0, 1, ...,$ we have:
\begin{eqnarray}
&&\hspace{-.4in} \Delta(t_k)\leq CT   \nonumber \\
&&\hspace{-.7in} \qquad - \sum_{t=0}^{T-1}\expect{\sum_{n}Q_{n}(t_k+t)\big[  L_n(k, t) - L^{\tsf{av}}_n\big]\left.|\right. \bv{Q}(t_k)}.\hspace{-.05in}  \label{eq:drift-bound} 
\end{eqnarray}
Here $C=\frac{1}{2}\sum_{n}\big([L_n^{\tsf{max}}]^2 + [L_n^{\tsf{av}}]^2\big)$.  $\Diamond$
\end{lemma}
\begin{proof}
See Appendix C. 
\end{proof}

To construct our algorithm, we choose a parameter $\epsilon>0$, which represents how close we want our performance to be to the optimal value. We then define $\eta\triangleq1/\epsilon$ and substract from both sides of (\ref{eq:drift-bound}) the term $\eta \expect{\text{Welfare}(k) \left.|\right. \bv{Q}(t_k)}$ to get: 
\begin{eqnarray*}
&& \Delta(t_k) -\eta\expect{\text{Welfare}(k) \left.|\right. \bv{Q}(t_k)} \\
&&\,\,\, \leq CT -\eta\expect{\text{Welfare}(k) \left.|\right. \bv{Q}(t_k)}   \\
&& \hspace{-.45in} \qquad\quad\,\,\,- \sum_{t=0}^{T-1}\expect{\sum_{n}Q_{n}(t_k+t)\big[  L_n(k, t) - L^{\tsf{av}}_n\big] \left. |\right. \bv{Q}(t_k)}.
\end{eqnarray*}
Now we plug in the definition of $\text{Welfare}(k)$ in  (\ref{eq:welfare-def})  to get:
\begin{eqnarray}
\hspace{-.3in}&& \Delta(t_k) -\eta\expect{\text{Welfare}(k) \left.|\right. \bv{Q}(t_k)} \nonumber   \\
\hspace{-.3in}&&\,\,\, \leq CT -\eta\sum_{t=0}^{T-1}\expect{\sum_{n=1}^N U_{n}(L_n(k, t), t) - \text{Cost}(k, t)\left.|\right. \bv{Q}(t_k)} \nonumber \\
\hspace{-.3in}&& \hspace{-.45in} \qquad\quad\,\,\,- \sum_{t=0}^{T-1}\expect{\sum_{n}Q_{n}(t_k+t)\big[  L_n(k, t) - L^{\tsf{av}}_n\big] \left.|\right. \bv{Q}(t_k)}. \label{eq:drift-utility-0}
\end{eqnarray}
We can now construct our algorithm by choosing the actions that  minimize the RHS of (\ref{eq:drift-utility-0}). However, directly doing so will require solving the following dynamic programming problem:
\begin{eqnarray}
\hspace{-.3in}&&\Phi(\bv{p}(k))\triangleq\max:   \sum_{t=0}^{T-1}\expect{\sum_{n=1}^N \eta U_{n}(L_n(k, t), t) - \eta\text{Cost}(k, t)\nonumber\\
\hspace{-.3in}&&\qquad\qquad\qquad\qquad\,\,\,  + \sum_nQ_{n}(t_k+t)  L^{\tsf{d}}_n(p(k, t), t) \left.|\right. \bv{Q}(t_k)} \nonumber\\
\hspace{-.3in}&&  \hspace{-.8in} \qquad\qquad \qquad \text{s.t.} \quad  \bv{p}(k)\in\script{P}. \label{eq:wma-max}
\end{eqnarray}
Here $\bv{p}(k) = (p(k, t), t=0, ..., T-1)$ is the price vector of day $k$, and $\script{P}$ is the set of feasible prices defined in Section \ref{section:model}. 

%
Therefore, we instead use an alternative method to \emph{approximately} solve this dynamic programming problem, in which case our actions approximately minimize the RHS of (\ref{eq:drift-utility-0}). 
The intuition behind our approach is that, since the queue sizes will only change by a finite amount in any slot,  instead of using the exact deficit queue value in very time slot, we use the  ``less updated''   information $\bv{Q}(t_k)$. By doing so, we can  still approximately minimize the RHS of (\ref{eq:drift-utility-0}). 
However, this approximation not only enables us to decouple the pricing decisions across time slots, which greatly reduces the computational complexity of our algorithm, but also allows us to solve the pricing problem ((\ref{eq:wma-max-approx}) defined below) in a distributed manner using using the ``consistency price''  approach in  \cite{chiang-consistency} in a distributed manner without the users submitting their utility functions to the utility company. This makes our algorithm very suitable for practical implementation. 

%

%
Below, for the ease of presentation, we denote both $L_n(k, t)$ and $U_n(L, t)$ as random functions of the price for convenience: 
\begin{eqnarray}
L^{\tsf{d}}_n(k, t) = L^{\tsf{d}}_n(p(k, t), t), \,\, U_n(L, t) = U_n(p(k, t), t). 
\end{eqnarray}

\underline{\tsf{Welfare Maximization Algorithm (WMA):}} 
\begin{itemize}
\item \underline{\textbf{On Day $k-1$:}} the utility company performs pricing and and the users perform consumption planning for day $k$: 

\begin{itemize}
\item \emph{Utility-Company-Pricing:} Based on its knowledge of the user demand, the utility company chooses the price $p(k, t)$ for time $t_k + t$, so as to maximize: 
\begin{eqnarray}
\hspace{-.7in}&&\Phi^{\tsf{A}}(p(k, t))\triangleq\max:   \expect{\sum_{n=1}^N \eta U_{n}(p(k, t), t) - \eta\text{Cost}(k, t)\nonumber\\
\hspace{-.7in}&&\qquad\qquad\qquad\qquad\quad  + \sum_nQ_{n}(t_k)  L^{\tsf{d}}_n(p(k, t), t) \left.|\right. \bv{Q}(t_k)} \nonumber\\
\hspace{-.7in}&& \qquad\qquad\qquad \text{s.t.} \quad \bv{p}(k) \in \script{P}. \label{eq:wma-max-approx}
\end{eqnarray}


\item \emph{Utility-Company-Day-Ahead-Power-Procurement:} Observe $(\bv{\beta}(k), \bar{\bv{\alpha}}(k))$, procure base-power $B(k, t)$ according to (\ref{eq:opt-B-def}) using knowledge of $L(k, t)$ and $Z(t)=X(t)-w(t)$. 

\item \emph{User-Consumption-Planning:} Choose the load $L^{\tsf{d}}_n(k, t)$  according to (\ref{eq:user-choice}), i.e.,  
\begin{eqnarray}
\hspace{-.7in}&&L^{\tsf{d}}_n(k, t)= \min\argmax{L_n^{\tsf{min}}(t)\leq L\leq L_n^{\tsf{d, max}}} \expect{U_n(L+w_n(k, t), t) \label{eq:user-choice-recap}\\
\hspace{-.7in}&&\qquad\qquad\qquad\qquad\qquad\qquad\quad  - p(k, t) (L+w_n(k, t))}.\nonumber
\end{eqnarray}

\end{itemize}

\item \underline{\textbf{On Day $k$:}} The user consumes power, and 
the utility company performs real-time power balancing and deficit queue update:

\begin{itemize}
\item \emph{User-Power-Consumption:} Each user consumes an amount of power $L^{\tsf{d}}_n(n, k)+w_n(k, t)$. 
\item \emph{Real-Time-Power-Purchasing-and-Deficit-Update:} The utility company procures power from the real-time market if necessary,  and updates the deficit queues according to (\ref{eq:queue-dynamic}).  $\Diamond$
\end{itemize}
\end{itemize}
%
Note that to implement the \tsf{WMA} algorithm, the utility company only has to use the instant value of market prices $(\bv{\beta}(k), \bar{\bv{\alpha}}(k))$, and does not need to know their distribution. This greatly simplifies the implementation of the algorithm and makes it more robust. 
%
%



\section{Performance Analysis}\label{section:analysis}
In this section, we analyze the performance of   \tsf{WMA}. We first have the following theorem, which states that the optimal social welfare can be achieved by a stationary and randomized algorithm $\Pi^*$ of the following structure: on any day $k-1$, the utility company observes the day-ahead and real-time price vector $(\bv{\beta}(k), \bar{\bv{\alpha}}(k))$, and chooses the price vector  i.i.d. from some countable subset $\{(p_{mt}^{(\bv{\beta}, \bar{\bv{\alpha}})}, t=0, ..., T-1), m=1, ..., \infty\}$ of $\script{P}$ according to some distribution $\{(a_{mt}^{(\bv{\beta}, \bar{\bv{\alpha}})}, t=0, ..., T-1), m=1, ..., \infty\}$, both potentially depend on $(\bv{\beta}(k), \bar{\bv{\alpha}}(k))$. 
Then, under any prices, the utility company procures the base-power according to (\ref{eq:opt-B-def}). 

\begin{theorem}\label{theorem:optimal-policy}
The optimal system welfare $\text{Welfare}^*_{\tsf{av}}$ is given by the following optimization problem. 
\begin{eqnarray}
\hspace{-.55in}&&\quad\,\,\text{Welfare}_{\tsf{av}} =\sup \big\{\text{Utility}_{\tsf{av}} - \text{Cost}_{\tsf{av}}\}  \label{theorem-eq-obj}\\
\hspace{-.55in}&& \text{s.t.} \,\, \,\,\text{Utility}_{\tsf{av}} =  \expect{ \sum_{t=0}^{T-1}\sum_{m=1}^{\infty}a_{mt}^{\bv{(\beta}, \bar{\bv{\alpha}})} \sum_n\expectm{U_n(p_{nmt}^{(\bv{\beta}, \bar{\bv{\alpha}})}, t)} },  \label{theorem-eq-utility}\\
\hspace{-.55in}&&\qquad\,\,\, \text{Cost}_{\tsf{av}} = \expect{ \sum_{t=0}^{T-1}\sum_{m=1}^{\infty}a_{mt}^{(\bv{\beta}, \bar{\bv{\alpha}})} \sum_n\expectm{\text{Cost}(p_{nmt}^{(\bv{\beta}, \bar{\bv{\alpha}})}, t)} },  \label{theorem-eq-cost}\\
\hspace{-.55in}&&\qquad\,  \quad TL_n^{\tsf{av}}\leq \expect{\sum_{t=0}^{T-1}\sum_{m=1}^{\infty}a_{mt}^{(\bv{\beta}, \bar{\bv{\alpha}})} \sum_n L_n^{\tsf{d}}(p_{nmt}^{(\bv{\beta}, \bar{\bv{\alpha}})}, t)},\,\forall\, n,  \label{theorem-eq-qou}\\
\hspace{-.55in}&&\qquad\,\,\,\, p_{nmt}^{(\bv{\beta}, \bar{\bv{\alpha}} )} = p_{mt}^{(\bv{\beta}, \bar{\bv{\alpha}} )},\,\,\forall\, n, m, t,  \label{theorem-eq-price}\\
\hspace{-.55in}&&\qquad\,\,\,\, a_{mt}^{(\bv{\beta}, \bar{\bv{\alpha}})}\geq0, \sum_{m=1}^{\infty}a_{mt}^{(\bv{\beta}, \bar{\bv{\alpha}})}=1, \,\forall\, t,  (\bv{\beta}, \bar{\bv{\alpha}}),\\
\hspace{-.55in}&&\qquad\,  (p_{nmt}^{(\bv{\beta}, \bar{\bv{\alpha}})}, t=0, ..., T-1) \in\script{P}, \,\forall\, (\bv{\beta}, \bv{\alpha}), m, n \\ 
\hspace{-.55in}&&\qquad\,  (p_{mt}^{(\bv{\beta}, \bar{\bv{\alpha}})}, t=0, ..., T-1) \in\script{P}, \,\forall\, (\bv{\beta}, \bv{\alpha}), m. 
\end{eqnarray}
Here $\sup\{\}$ denotes the supremum. The outside expectations are taken over the day-ahead and real-time price vectors $(\bv{\beta}(k), \bar{\bv{\alpha}}(k))$. The inside expectations are taken over the user usage variations $w_n(t)$. $\Diamond$
\end{theorem}
\begin{proof}
The proof is done by using Caratheodory's theorem and is similar to \cite{neelynowbook}.  Hence, the details are omitted for brevity. 
\end{proof}

We see that the constraint (\ref{theorem-eq-price}) requires that  \emph{ the prices to all the users are the same for all time.} From this result, we get the following corollary, which characterizes the welfare loss due to the equal price constraint, which we call \emph{price-of-single-price}, and denote as $\tsf{PoSP}$. 
\begin{coro}
Denote $\text{Welfare}^{\tsf{diff}}_{\tsf{av}}$ the optimal social welfare when the  utility company's prices can be set differently for each individual user, and $\text{Welfare}^{\tsf{r}}_{\tsf{av}}$ the optimal value of  of (\ref{theorem-eq-obj}) without the constraint (\ref{theorem-eq-price}). 
Then, we have: 
\begin{eqnarray}
\text{Welfare}^{\tsf{diff}}_{\tsf{av}} = \text{Welfare}^{\tsf{r}}_{\tsf{av}}\geq \text{Welfare}^{*}_{\tsf{av}}, 
\end{eqnarray}
and the \emph{price-of-linear-pricing} is given by: 
%
\begin{eqnarray}
\tsf{PoSP} = \text{Welfare}^{\tsf{r}}_{\tsf{av}} - \text{Welfare}^*_{\tsf{av}}. \Diamond\label{eq:plp}
\end{eqnarray}
\end{coro}

We can now present the performance results of the \tsf{WMA} algorithm in the following theorem. 
\begin{theorem}\label{theorem:wma-per} (\tsf{WMA} with i.i.d. market prices) 
The average welfare achieved by \tsf{WMA} with $\eta=1/\epsilon$ satisfies: 
\begin{eqnarray}
\hspace{-.2in}\text{Welfare}_{\tsf{av}}^{\tsf{WMA}} \geq \text{Welfare}^*_{\tsf{av}} - C_1T/\eta. \label{eq:welfare-bound}
\end{eqnarray}
Here $C_1=\frac{T}{2}\sum_{n}\big([L_n^{\tsf{max}}]^2 + [L_n^{\tsf{av}}]^2\big)$ is a constant independent of $\eta$. 
Moreover, for all $t_k, k\geq0$ and $t$, we have: 
\begin{eqnarray}
\hspace{-.2in}\sum_{n=1}^NQ_n(t_k+t) \leq   \delta_{\tsf{max}} N\gamma^2\eta  +T\sum_n L_n^{\tsf{av}}.\Diamond\label{eq:queue-bound}
\end{eqnarray}
\end{theorem}

We first see from Theorem \ref{theorem:wma-per} that the average welfare achieved by  \tsf{WMA} is within $O(\epsilon)$ of the maximum. Thus, as we decrease the value of $\epsilon$ (or equivalently, increase the value of $\eta$), we can push the average welfare to be arbitrarily close to the maximum value. On the other hand, we also see from (\ref{eq:queue-bound}) that the deficit queues are \emph{deterministically} upper bounded, implying that the users' QoU will be guaranteed with \emph{bounded} deficit. 
Indeed, each user roughly has to defer an aggregate load of $\delta_{\tsf{max}}\gamma^2\eta+TL_n^{\tsf{av}}$. This provides explicit delay guarantees to the users,  and is very useful in practice. 
Note that  the bound (\ref{eq:queue-bound})  is indeed a worst case \emph{upper} bound.  As we will see in Section \ref{section:sim}, the actual average deficit queue size can be much smaller. 

Below we prove the queueing bound (\ref{eq:queue-bound}), the proof of (\ref{eq:welfare-bound}) will be presented in Appendix D. 

\begin{proof} (Theorem \ref{theorem:wma-per}) 
To prove (\ref{eq:queue-bound}), we show that it holds for every time $t_k, k\geq1$, using induction. 
First, we see that it holds for $k=1$ because $\sum_nQ_n(0)=0$ and it increases by at most $\sum_nL_n^{\tsf{av}}$ every time slot. Hence $\sum_nQ_n(t_1)\leq T\sum_{n}L_n^{\tsf{av}}$. 

Now suppose  (\ref{eq:queue-bound}) holds for $k=1, ..., K$, we prove that it also holds for $k=K+1$. 

(I) First, if $\sum_nQ_{n}(t_{K})\leq \eta \delta_{\tsf{max}} N\gamma^2$, then (\ref{eq:queue-bound}) holds for $t_{K+1}$. This is because the maximum increase of $\sum_nQ_n(t)$ over any day is $T\sum_nL_n^{\tsf{av}}$. 

(II) Now suppose $\sum_nQ_{n}(t_{K})> \eta \delta_{\tsf{max}} N\gamma^2$. We will show that the term $\expectm{-\eta\text{Cost}(K, t)+\sum_nQ_n(t_K+t)L^{\tsf{d}}_n(k, t)}$ is an increasing function of $L^{\tsf{d}}_n(K, t)$ for all $t$ on day $K$.  
Thus,  \tsf{WMA} will encourage the users to consume more loads, which will then reduce the deficit queue sizes. 
To do so, we first see by Lemma \ref{lemma:value-renewable} that:  
\begin{eqnarray} 
\expectm{\eta\text{Cost}(K,  t)} \leq \eta\delta_{\tsf{max}} \sum_nL^{\tsf{d}}_n(K, t).  \label{eq:cost-bounded} 
\end{eqnarray}
%
%
Now we fix a user $n^*$ with load $L^{\tsf{d}}_{n^*}(K, t)$. If $\sum_nQ_n(t_K)L^{\tsf{d}}_n(K, t)> \eta\delta_{\tsf{max}} N\gamma^2$, we can use (\ref{eq:hetero}) to get: 
\begin{eqnarray}
\sum_nQ_n(t_K)L^{\tsf{d}}_n(K, t) &\geq& \sum_nQ_n(t_K) L^{\tsf{d}}_{n^*}(K, t)/\gamma\nonumber\\
&>& \eta\delta_{\tsf{max}} N\gamma^2 L^{\tsf{d}}_{n^*}(K, t)/\gamma\nonumber\\
&>& \eta\delta_{\tsf{max}} \sum_{n} L^{\tsf{d}}_n(K, t). \label{eq:sum-load-bdd}
\end{eqnarray}
%
%
This shows that the term $\expectm{-\eta\text{Cost}(K, t)+\sum_nQ_n(t_K)L^{\tsf{d}}_n(K, t)}$ is increasing in $L^{\tsf{d}}_n(K, t)$ throughout day $K$, which implies that the utility company will set the price to $0$, and the users will always chooses $L^{\tsf{d}}_n(K, t)=L_n^{\tsf{d, max}}$. Since  $L_n^{\tsf{d, max}}\geq L_n^{\tsf{av}} + w_n^{\tsf{max}}$, the deficit queue sizes will not further increase on day $K$. This shows that $\sum_nQ_n(t_{K})\geq\sum_nQ_n(t_{K+1})$, which by induction completes the proof of (\ref{eq:queue-bound}). 
\end{proof}

Now we present the performance results of \tsf{WMA} under more general market price processes. 
\begin{theorem}\label{theorem:wma-per-markov} (\tsf{WMA} with Markov market prices) 
Suppose $(\bv{\beta}(k), \bar{\bv{\alpha}}(k))$ evolves according to some finite state irreducible and aperiodic Markov chain. 
Then,  \tsf{WMA} with $\eta=1/\epsilon$ achieves: 
\begin{eqnarray}
\hspace{-.2in}\text{Welfare}_{\tsf{av}}^{\tsf{WMA}} &\geq& \text{Welfare}^*_{\tsf{av}} - O(\epsilon), \label{eq:per-bdd-markov}\\
\hspace{-.2in}\sum_{n=1}^NQ_n(t_k+t) &\leq&   \delta_{\tsf{max}} N\gamma^2\eta  +T\sum_n L_n^{\tsf{av}}.\label{eq:q-bdd-markov}\Diamond
\end{eqnarray}
\end{theorem}
\begin{proof} (Theorem \ref{theorem:wma-per-markov})
First we see that the queueing bound (\ref{eq:q-bdd-markov}) can be proven using the exact same argument as in the proof of Theorem \ref{theorem:wma-per}. The performance bound (\ref{eq:per-bdd-markov}) can be proven using a multiple-slot drift argument developed in \cite{huangneely_qlamarkovian}. The details are omitted for brevity. 
\end{proof}

\section{Simulation}\label{section:sim}
In this section, we present simulation results for the \tsf{WMA} algorithm. 
In the simulation, we use $T=24$ and consider there are $2$ users. Note that here each user can be viewed as representing a class of consumers. 
%
We express the user load in unit of $100$MW. \footnote{The number of users and the following per-user load parameters are chosen to match the annual average hourly power demand of PG$\&$E for the whole year $2010$, which is on the order of $4000$ MW. } 
We assume that $L_n^{\tsf{min}}(t)=3$ if $t\notin[9, 18]$, and $L_n^{\tsf{min}}(t)=5$ if $t\in[9, 18]$ for all users, and that $L_n^{\tsf{max}}=12$. %
We assume that both users have the following utility function, which have units of  $1000$  dollars per $100$ MW:  
\begin{eqnarray}
\hspace{-.5in}&&U_n(L, t)  = \left\{\begin{array}{rl}
8L &  t\in[9, 18], \, L \leq 5\\
0.8L+36 & t\in[9, 18], \, L [5, 6]\\ 
4L + 16.8, & t\in[9,18], \, L\in[6, 12], \\ 
a\min( L,  6) , & t\notin[9, 18], \, L\leq12. \label{eq:utility-base} 
\end{array}\right. 
\end{eqnarray}
Here $a$ is tunable parameter which allows us to study how the off-peak utility affects the social welfare. 
Since both users have the same utility function, we have $\gamma=1$. 
We note that the utility functions for time slots in $[9, 18]$ are non-concave. 
We assume that  the users have different QoU requirements, i.e., $L_1^{\tsf{av}} = 4.5$ and $L_2^{\tsf{av}}=8$. Thus, user $1$ represents consumers that are more flexible in power consumption, and are more willing to shift their load. 
We also assume for simplicity that $w_n(t)=0$ for all $t$. Thus, each user consumes exactly the intended load. 
%
The  feasible price set is chosen to be $\script{P} = [0, 800]$, which corresponds to $0\cent$ to  $8\cent$ per kWh, and is consistent with the current electricity market prices. 

We also use the following real-world data for market prices and renewable energy. 

\textbf{Hourly market prices:} We first compute $12$ hourly price data sets, for day-ahead and real-time market prices by averaging the data of PG$\&$E in the  year $2010$ (One for each month. Data are from the CAISO daily report on the Federal Energy Regulatory Commission (FERC) website \cite{caiso-web}). Then, on each day, we independently choose one price vector pair from the $12$ data set uniformly at random. We convert the prices to thousand dollars per $100$MW. 
 
\textbf{Hourly renewable energy:} We then compute the hourly renewable available energy  using the recorded wind power of $100$ wind turbines located near $40.42N, 124.39W$ at  the west coast in California in year $2006$, which has an aggregate capacity of $300$MW (data from the National Renewable Energy Laboratory (NREL) website \cite{wind-web}). 
For each hour, we first average the values of the recorded wind power of that hour (sampled every $10$ minutes) to obtain $365$ samples (normalized to with unit  $100$MW). We then use the samples as the state space and use the empirical distribution as its true distribution. 

\begin{figure}[cht]
\vspace{-.1in}
\centering
\includegraphics[height=2.5in, width=3.5in] {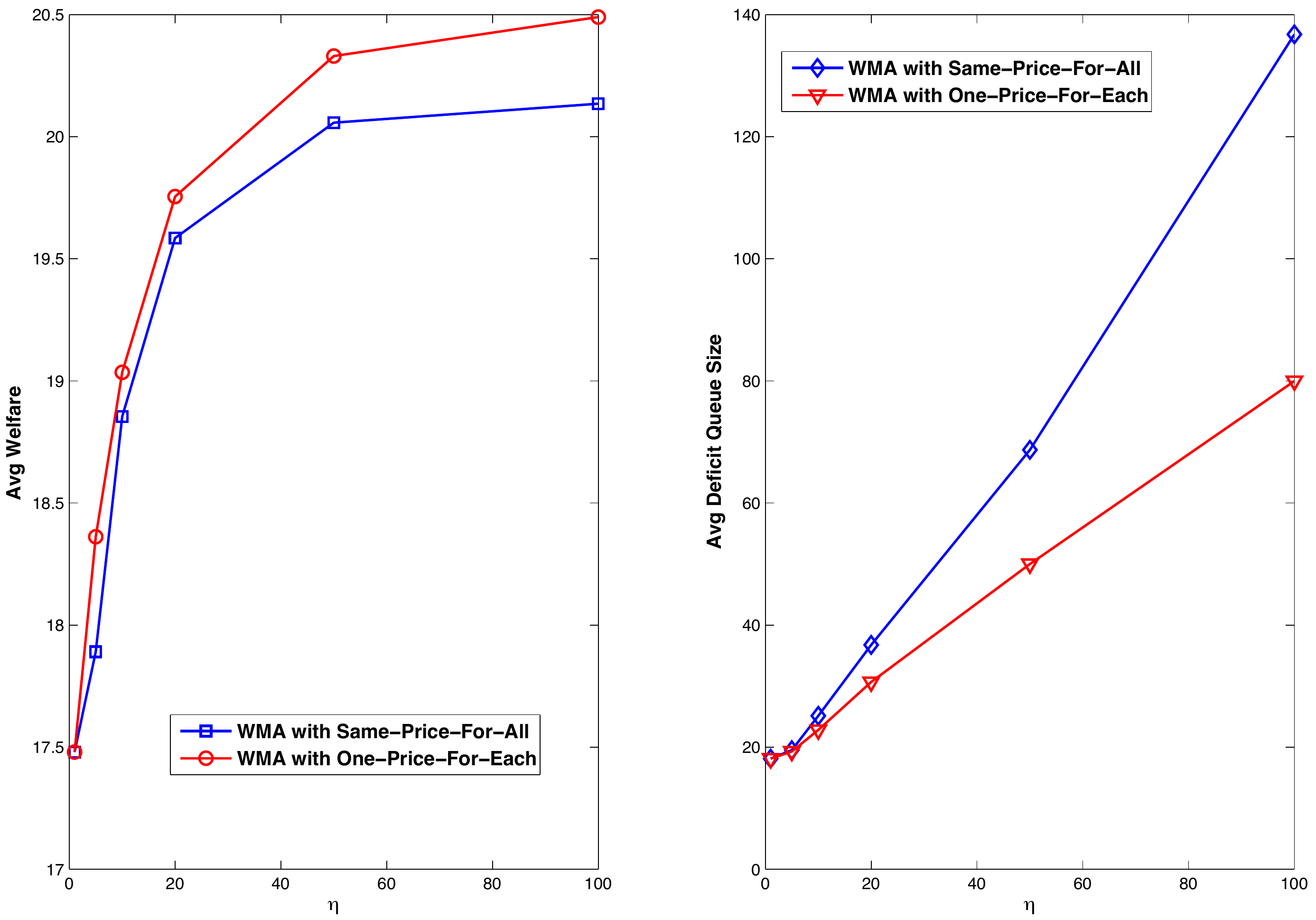}
\vspace{-.2in}
\caption{\small{Average achieved welfare and average deficit queue size. The off-peak utility coefficient $a=3$}. }
\label{fig:welfare1}
\vspace{-.1in}
\end{figure}

We first simulate our algorithm with the off-peak utility coefficient $a=3$. 
The simulation results are plotted in Fig. \ref{fig:welfare1} and \ref{fig:welfare2}. 
We simulated \tsf{WMA} with both two pricing schemes: same-price-for-all and one-price-for-each (Note that in this case \tsf{WMA} will also converge to the corresponding optimal  when different prices are allowed). 
From the results, we see that \tsf{WMA} efficiently improves the social welfare while effectively controlling the QoU. 
For instance, in Fig. \ref{fig:welfare1}, when $\eta=20$, the social welfare (under the same-price-for-all scheme) already reaches $19.6$ thousand dollars per hour, which is near-optimal. Also, the corresponding deficit is roughly  $37$, which  roughly corresponds to a $3$ hour delay of the load, because the average hourly load is $5+8=13$. 
We note that our deficit queueing bound (\ref{eq:queue-bound}) for $\eta=20$ is $756$. This suggests that the actual average deficit queue size of \tsf{WMA} can be much smaller than that in (\ref{eq:queue-bound}). 

From the simulation results, we can also see  how the pricing scheme affects the social welfare. 
We see that  this price diversity not only allows us to improve the social welfare by $2\%$ (Fig. \ref{fig:welfare1}) to $9\%$ (Fig. \ref{fig:welfare2}), but also results in a smaller deficit (up to $41\%$ saving in both cases). This is because allowing different prices enables the system to better adapt to the deficits, which ensures the QoU being met more quickly. 

\begin{figure}[cht]
\centering
\vspace{-.1in}
\includegraphics[height=2.5in, width=3.5in] {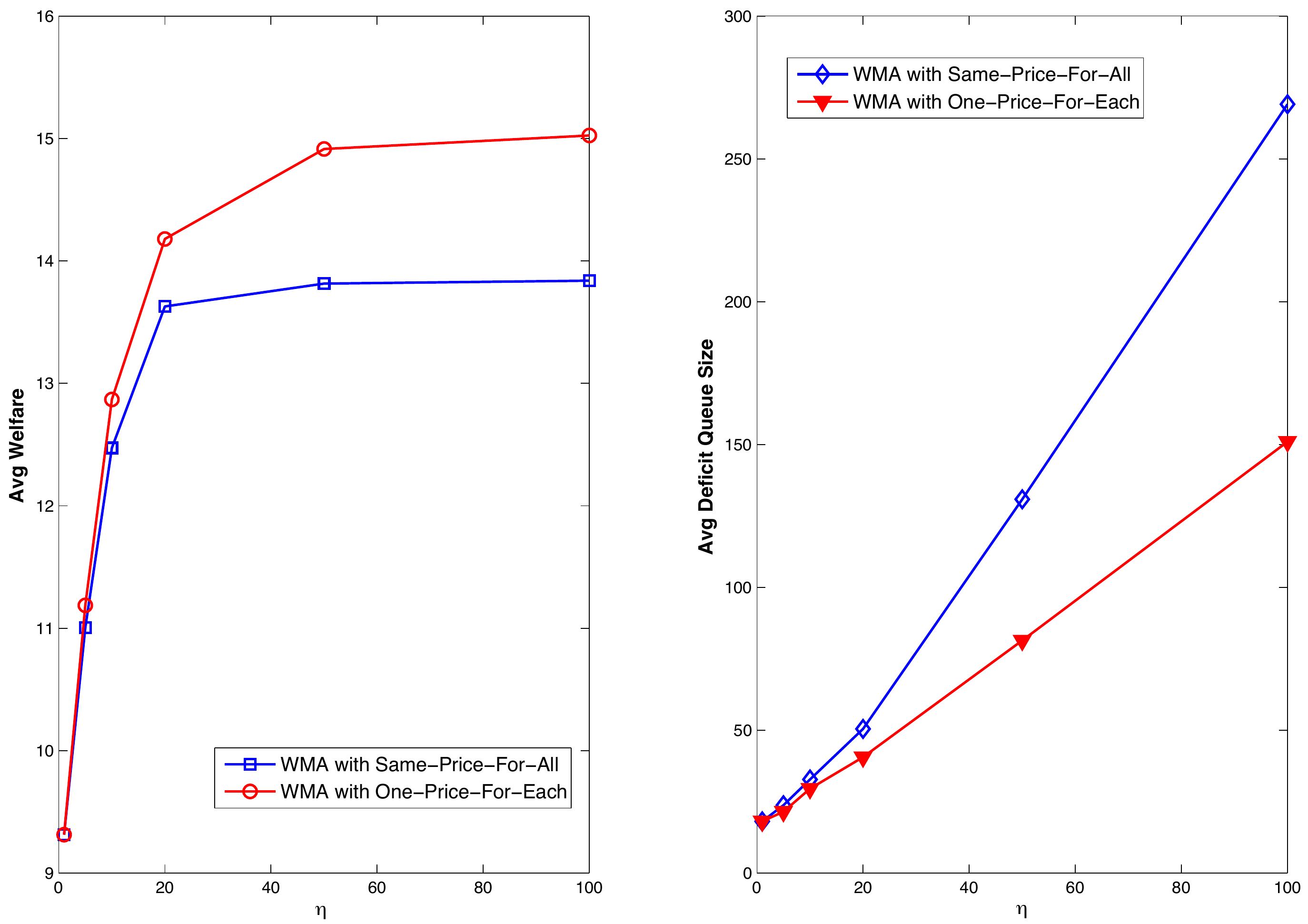}
\vspace{-.2in}
\caption{\small{Average achieved welfare and average deficit queue size. The off-peak utility coefficient $a=2$}. }
\label{fig:welfare2}
\vspace{-.2in}
\end{figure}




\section{Conclusion}\label{section:conclusion}
In this paper, we consider the problem of jointly optimizing power procurement and  demand response to maximize social welfare. We develop  a low-complexity algorithm called the \emph{welfare maximization algorithm} (\tsf{WMA}), which  is constructed based on a two-timescale Lyapunov optimization technique. We prove that \tsf{WMA} achieves a close-to-optimal utility and ensures that the users' average power usages  meet their requirements with bounded deficit. We also perform simulation using real-world data and show that \tsf{WMA} effectively achieves near-optimal social welfare. 

\section*{Appendix A -- Proof of Lemma \ref{lemma:value-renewable}}
Here we prove Lemma \ref{lemma:value-renewable}. 
We first recall the $\text{Cost}(k, t)$ function: 
\begin{eqnarray}
  \text{Cost}(k, t) \triangleq  \beta(k, t) B(k, t) + \alpha(k, t) Y(k, t),  \nonumber 
\end{eqnarray}
where the power deficit $Y(t)$ is given by: 
\begin{eqnarray}
Y(k, t) &=& \big[L(k, t) - X(k, t) - B(k, t)\big]^+ \nonumber\\
&=&  \big[L^{\tsf{d}}(k, t) - Z(k, t)- B(k, t)\big]^+. \label{eq:YasZ}
\end{eqnarray}
%
We now present the proof. For notation simplicity, we  omit the $k, t$ indexes when it is clear. 
\begin{proof} (Lemma  \ref{lemma:value-renewable}) 
First, it is clear that if $\bar{\alpha}(k, t)=\mathbb{E}[\alpha(k, t)]<\beta(k, t)$, then $B^*(k, t)=0$. Now suppose $\bar{\alpha}(k, t)\geq\beta(k, t)$. 
%
Using (\ref{eq:YasZ}) and omitting the $k, t$ indexes, we have: 
\begin{eqnarray}
\hspace{-.1in}\expectm{\text{Cost}(k, t)} = \bar{\alpha}\int_{-\infty}^{\infty} \big[L^{\tsf{d}} - B - z\big]^+f_Z(z, t)dz + \beta B\nonumber\\
\hspace{-.1in}=   \bar{\alpha}\int_{-\infty}^{L^{\tsf{d}} - B} \big[L^{\tsf{d}} - B - z\big]f_Z(z, t)dz + \beta B. \label{eq:opt-G}
\end{eqnarray}
It can be verified that $\expectm{\text{Cost}(k, t)}$ is convex in $B$. Thus, to minimize it, we take the derivate to get:
\begin{eqnarray}
\hspace{-.3in}&&\frac{d\expectm{\text{Cost}}}{dB} = \beta + \bar{\alpha}\big[- (L^{\tsf{d}}-B)f_Z(L^{\tsf{d}}-B, t)\big] \nonumber\\
\hspace{-.3in}&&\qquad\quad - \bar{\alpha} \bigg[ -(L^{\tsf{d}}-B)f_Z(L^{\tsf{d}}-B, t) - \int_{-\infty}^{L^{\tsf{d}}-B}f_Z(z, t)dz \bigg] \nonumber\\
\hspace{-.3in} &&\qquad = \beta - \bar{\alpha} F_Z(L^{\tsf{d}}-B, t). \label{eq:opt-cond}
\end{eqnarray}
Setting (\ref{eq:opt-cond}) to zero, we obtain:
\begin{eqnarray}
B^* = \bigg[L^{\tsf{d}}-F_Z^{-1}(\frac{\beta}{\bar{\alpha}}, t)\bigg]^+. \label{eq:opt-B}
\end{eqnarray}
This proves (\ref{eq:opt-B-def}). To see (\ref{eq:value-renewable}), suppose $B^*>0$ and plug (\ref{eq:opt-B}) back into (\ref{eq:opt-G}), we get:
\begin{eqnarray}
\hspace{-.4in}&&  \expectm{\text{Cost}(k, t)} = \beta B+\bar{\alpha} (L^{\tsf{d}} - B)\int_{-\infty}^{F_Z^{-1}(\frac{\beta}{\bar{\alpha}}, t)} f_Z(z, t)dz  \\
\hspace{-.4in}&& \qquad\qquad\qquad\qquad\qquad- \bar{\alpha} \int_{-\infty}^{F_Z^{-1}(\frac{\beta}{\bar{\alpha}}, t)} zf_Z(z, t)dz  \nonumber\\
\hspace{-.4in}&& \qquad\qquad\quad\,\,\,\, = \beta L^{\tsf{d}} - \bar{\alpha} \int_{-\infty}^{F_Z^{-1}(\frac{\beta}{\bar{\alpha}}, t)} zf_Z(z, t)dz. 
\end{eqnarray}
Here the second equality uses the fact that $\int_{-\infty}^{F_Z^{-1}(\frac{\beta}{\bar{\alpha}}, t)} f_Z(z, t)dz=\beta/\bar{\alpha}$. 
Finally, it can be seen from (\ref{eq:cost-zero-B}) and  (\ref{eq:value-renewable}) that $\expectm{\text{Cost}(k, t)}$ is non-decreasing in $L^{\tsf{d}}_n(k, t)$. This completes the proof.  
\end{proof}

\section*{Appendix  B -- Proof of Lemma \ref{lemma:value-of-renewable}}
Here we prove Lemma \ref{lemma:value-of-renewable}. 

\begin{proof} (Lemma \ref{lemma:value-of-renewable})
Since $w(t)$ is independent of $X(t)$ and has zero mean, it suffices to show that $\text{VoR}(t)$ is increasing with the mean of $Z(t)$ and non-increasing with the variance of $Z(t)$. 
To begin, we note that $Z(t)$ can be written as a sum of a  constant, which determines its mean, and a zero mean random variable, i.e., (we drop the time indexes below for convenience): 
\begin{eqnarray}
Z=\mu+ \sigma H. 
\end{eqnarray}
Here $H$ has unit variance, zero mean,  and a pdf $f_H(h)$. $\sigma^2$ is the variance of $Z$. 

We first consider the case when $B^*\geq0$.  Recall (\ref{eq:value-renewable}): 
\begin{eqnarray}
\text{VoR}(t) = \bar{\alpha} \int_{-\infty}^{F_Z^{-1}(\frac{\beta}{\bar{\alpha}})} zf_Z(z)dz. 
\end{eqnarray}
Denote $\theta=F_Z^{-1}(\frac{\beta}{\bar{\alpha}})$, we can rewrite the above as: 
\begin{eqnarray*}
\text{VoR}(t) &=&\bar{\alpha}\int_{-\infty}^{\theta} zf_Z(z)dz\\
 &=& \bar{\alpha} \int_{-\infty}^{\frac{\theta-\mu}{\sigma}}(\mu+\sigma h)f_H(h)dh\\
&=& \bar{\alpha} \int_{-\infty}^{\frac{\theta-\mu}{\sigma}}\mu f_H(h)dh + \sigma\bar{\alpha} \int_{-\infty}^{\frac{\theta-\mu}{k}}hf_H(h)dh\\
&=& \bar{\alpha} \mu F_H(\frac{\theta-\mu}{\sigma})+\sigma\bar{\alpha} \int_{-\infty}^{\frac{\theta-\mu}{\sigma}}hf_H(h)dh.
\end{eqnarray*}
Now since $F_Z(\theta)=\frac{\beta}{\bar{\alpha}}$, we see that $F_H(\frac{\theta-\mu}{\sigma})=\frac{\beta}{\bar{\alpha}}$. Thus $\frac{\theta-\mu}{\sigma}=F_H^{-1}(\frac{\beta}{\bar{\alpha}})$. 
Denote $\theta_H=F_H^{-1}(\frac{\beta}{\bar{\alpha}})$, then we have:
\begin{eqnarray}
\text{VoR}(t) =  \beta \mu +\sigma\bar{\alpha} \int_{-\infty}^{\theta_H}hf_H(h)dh.\label{eq:gain-mean-variance}
\end{eqnarray}
Now since $\expectm{H}=0$, we see that $\int_{-\infty}^{\theta_H}hf_H(h)dh\leq0$. Hence $\text{VoR}(t)$ is an increasing function of $\mu$ and a non-increasing function of $\sigma$. 

Now suppose $B^*=0$. In this case, we will show that $\expectm{\text{Cost}(k, t)}$ is decreasing in $\mu$ and increasing in $\sigma$. To this end, we have:
\begin{eqnarray*}
\hspace{-.3in}&&\expectm{\text{Cost}}= \bar{\alpha}  \int_{-\infty}^{L^{\tsf{d}} } (L^{\tsf{d}} -z) f_Z(z)dz  \\
\hspace{-.3in}&&\qquad\quad\,\,\,=  \bar{\alpha}   \int_{-\infty}^{\frac{L^{\tsf{d}} -\mu}{\sigma}} (L^{\tsf{d}} -\mu-\sigma h)f_H(h)dh \\
\hspace{-.3in}&&\qquad\quad\,\,\,= \bar{\alpha} (L^{\tsf{d}} -\mu)  \int_{-\infty}^{\frac{L^{\tsf{d}} -\mu}{\sigma}} f_H(h)dh - \bar{\alpha}  \sigma \int_{-\infty}^{\frac{L^{\tsf{d}} -\mu}{\sigma}} hf_H(h)dh. 
\end{eqnarray*}
We can now take derivative with respect to $\mu$ to get:
\begin{eqnarray}
\frac{\partial \expectm{\text{Cost}}}{d\mu} &=& -\bar{\alpha} F_H(\frac{L^{\tsf{d}}-\mu}{\sigma})  \label{eq:derivative-mu} \\
&\leq& 0. \nonumber 
\end{eqnarray}
This implies that $\expectm{\text{Cost}}$ is decreasing with $\mu$. We can also take a derivative with respect to $\sigma$ to get: 
\begin{eqnarray}
\frac{\partial \expectm{\text{Cost}}}{d\sigma} 
&=& -\bar{\alpha} \int_{-\infty}^{\frac{L^{\tsf{d}}-\mu}{\sigma}}hf_H(h)dh\label{eq:derivative-k}\\ 
&\geq& 0. \nonumber 
\end{eqnarray}
In the last step, we have used the fact that $\expectm{H}=0$, thus $ \int_{-\infty}^{\frac{L^{\tsf{d}}-\mu}{\sigma}}hf_H(h)dh\leq0$. 
We see that the function is increasing with $\sigma$. 
\end{proof}

\section*{Appendix C -- Proof of Lemma \ref{lemma:pure-drift}}
In this section, we prove Lemma \ref{lemma:pure-drift}. 
\begin{proof} (Lemma \ref{lemma:pure-drift}) Squaring both sides of the queueing dynamic equation (\ref{eq:queue-dynamic}), we have:
 \begin{eqnarray*}
[Q_n(t_k+t+1)]^2 &\leq& [Q_n(t_k+t)]^2 + [L_n(k, t)]^2 + [L_n^{\tsf{av}}]^2\\
&& - 2Q_n(k, t) \big[ L_n(k, t) - L_n^{\tsf{av}}  \big]. 
\end{eqnarray*}
Multiplying both sides with $\frac{1}{2}$, using the fact that $L_n(k, t)\leq L^{\tsf{max}}_n$,  and summing over all $n$, we have: 
 \begin{eqnarray*}
V(t_k+t+1) &\leq& V(t_k+t) + \frac{1}{2}\sum_{n}\bigg([L_n^{\tsf{max}}]^2 + [L_n^{\tsf{av}}]^2\bigg)\\
&& - \sum_nQ_n(t_k+t) \big[ L_n(k, t) - L_n^{\tsf{av}}  \big]. 
\end{eqnarray*}
Denote $C=\frac{1}{2}\sum_{n}\big([L_n^{\tsf{max}}]^2 + [L_n^{\tsf{av}}]^2\big)$. 
Summing over $t=0, ..., T-1$ and taking expectations on both sides conditioning on $\bv{Q}(t_k)$, we get the following:
 \begin{eqnarray}
&&\Delta(t_k) \leq CT \label{eq:drift-foo}\\
&&\qquad- \sum_{t}\expect{\sum_nQ_n(t_k+t) \big[ L_n(k, t) - L_n^{\tsf{av}}  \big] \left.|\right. \bv{Q}(t_k)}. \nonumber
\end{eqnarray}
This proves the lemma. 
\end{proof}

\section*{Appendix D -- Proof of (\ref{eq:welfare-bound}) of Theorem \ref{theorem:wma-per}}
In this section, we prove Theorem \ref{theorem:wma-per}. To proceed, we first have the following lemma, which shows that \tsf{WMA} approximately minimizes  (\ref{eq:wma-max}).  
\begin{lemma}\label{lemma:approx-dp}
Let $\Phi^*$ be the maximum value of $\Phi(\bv{p}(k))$ and let $\bv{p}^{\tsf{A}}(k)=(p^{\tsf{A}}(k, t), t=0, ..., T-1)$ be the maximizers of $\Phi^{\tsf{A}}(\bv{p}(k))\triangleq\sum_{\tau=0}^{T-1}\Phi^{\tsf{A}}(p(k, t))$. Then $\Phi^{\tsf{A}}(\bv{p}^{\tsf{A}}(k))\geq \Phi^*- TC_0$, where $C_0=(T-1)\sum_n\big( [L_n^{\tsf{max}}]^2 +  [L_n^{\tsf{av}}]^2\big)/2$. 
\end{lemma}
\begin{proof}  (Lemma \ref{lemma:approx-dp}) 
From (\ref{eq:queue-dynamic}), we have the following inequalities:
\begin{eqnarray*}
Q_n(t_k+t) \leq Q_n(t_k) + tL_{n}^{\tsf{av}}, \,\,\,
Q_n(t_k+t) \geq Q_n(t_k) - tL_{n}^{\tsf{max}}. 
\end{eqnarray*}
Using these inequalities, we have: 
\begin{eqnarray}
&&Q_n(t_k+t) \big[ L_n(k, t) - L_n^{\tsf{av}}  \big]  \label{eq:drift-queue-foo1}\\
&&\qquad \leq Q_n(t_k)  \big[ L_n(k, t) - L_n^{\tsf{av}}  \big] + t\big( [L_n^{\tsf{max}}]^2 +  [L_n^{\tsf{av}}]^2\big). \nonumber
\end{eqnarray}
Then, by comparing $\Phi(\bv{p}(k))$ and $\Phi^{\tsf{A}}(\bv{p}(k))$, we see that for any $\bv{p}(k)$, 
\begin{eqnarray}
\Phi(\bv{p}(k)) &\leq& \Phi^{\tsf{A}}(\bv{p}(k)) + \sum_{t=0}^{T-1}t\sum_n\big( [L_n^{\tsf{max}}]^2 +  [L_n^{\tsf{av}}]^2\big)\nonumber\\
&=&  \Phi^{\tsf{A}}(\bv{p}(k)) +\frac{T(T-1)}{2}\sum_n\big( [L_n^{\tsf{max}}]^2 +  [L_n^{\tsf{av}}]^2\big) \nonumber\\
&\leq& \Phi^{\tsf{A}}(\bv{p}^{\tsf{A}}(k)) +\frac{T(T-1)}{2}\sum_n\big( [L_n^{\tsf{max}}]^2 +  [L_n^{\tsf{av}}]^2\big). \nonumber
\end{eqnarray}
Since this holds for all $\bv{p}(k)$, it holds for the $\bv{p}(k)$ that maximizes $\Phi(\bv{p}(k))$. Thus, by defining $C_0\triangleq \frac{(T-1)}{2}\sum_n\big( [L_n^{\tsf{max}}]^2 +  [L_n^{\tsf{av}}]^2\big)$, we see that Lemma \ref{lemma:approx-dp} follows. 
\end{proof} 

We are now ready to prove Theorem \ref{theorem:wma-per}. 
\begin{proof} (Theorem \ref{theorem:wma-per})
From  Lemma \ref{lemma:approx-dp} we see that \tsf{WMA} minimizes the RHS of (\ref{eq:drift-utility-0}) to within $TC_0$ of the minimum. 
Thus we have the following inequality: 
\begin{eqnarray}
\hspace{-.25in}&& \Delta(t_k) -\eta\expect{\text{Welfare}^{\tsf{WMA}}(k) \left.|\right. \bv{Q}(t_k)}\label{eq:drift-utility-0-recap} \\
\hspace{-.25in}&&\,\,\, \leq C_1T -\eta\sum_{t=0}^{T-1}\expect{\sum_{n=1}^N U^{\tsf{A}}_{n}(L_n(k, t)) - \text{Cost}^{\tsf{A}}(k, t)\left.|\right. \bv{Q}(t_k)} \nonumber \\
\hspace{-.25in}&&\qquad\quad\,\,\,- \sum_{t=0}^{T-1}\expect{\sum_{n}Q_{n}(t_k+t)\big[  L^{\tsf{A}}_n(k, t) - L^{\tsf{av}}_n\big] \left.|\right. \bv{Q}(t_k)}. \nonumber
\end{eqnarray}
Here $C_1=C+C_0=\frac{T}{2}\big( [L_n^{\tsf{max}}]^2 +  [L_n^{\tsf{av}}]^2\big)$, and the superscript  \tsf{A} stands for any other alternative pricing policies, including the optimal randomized and stationary policy $\Pi^*$ defined in Theorem \ref{theorem:optimal-policy}. Now plugging the policy $\Pi^*$ into the RHS of (\ref{eq:drift-utility-0-recap}) and using the fact that:  
\begin{eqnarray}
\expect{\text{Welfare}^{\Pi^*}(k)} = \text{Welfare}^*_{\tsf{av}}, \,\,\,
\expect{L_n(k, t)} \geq L_n^{\tsf{av}}, \,\, \forall\, n, 
\end{eqnarray}
we immediately obtain:
\begin{eqnarray*}
\hspace{-.3in}&& \Delta(t_k) -\eta\expect{\text{Welfare}^{\tsf{WMA}}(k) \left.|\right. \bv{Q}(t_k)}  \leq C_1T -\eta\text{Welfare}^*_{\tsf{av}}. 
\end{eqnarray*}
Note that in the above step, we have use the fact that the policy $\Pi^*$ chooses actions purely as functions of the random system prices, and is independent of the deficit queue sizes. 
Now taking expectations on both sides over $\bv{Q}(t_k)$ and carrying out a telescoping sum over $k=0, 1, ..., K-1$, we get: 
\begin{eqnarray*}
\hspace{-.3in}&&\expect{L(t_{K}) - L(t_0)} -\eta\sum_{k=0}^{K-1}\expect{\text{Welfare}^{\tsf{WMA}}(k)}  \\
\hspace{-.3in}&&\leq KC_1T -K\eta\text{Welfare}^*_{\tsf{av}}. 
\end{eqnarray*}
Rearranging the terms, using the fact that $L(t)\geq0$ for all $t$,  and dividing both sides by $\eta K$, we get:
\begin{eqnarray*}
 \frac{1}{K}\sum_{k=0}^{K-1}\expect{\text{Welfare}^{\tsf{WMA}}(k)}  \geq \text{Welfare}^*_{\tsf{av}} -\frac{C_1T}{\eta} + \frac{\expectm{L(t_0)}}{K\eta}. 
\end{eqnarray*}
Taking the $\liminf$ as $K\rightarrow\infty$ and using the fact that $L(t_0)<\infty$, we see that (\ref{eq:welfare-bound}) follows. 
\end{proof}

$\vspace{-.2in}$
\bibliographystyle{unsrt}
\bibliography{../mybib}

\end{document}